\documentclass{amsart}

\usepackage[all]{xy}

\newcommand*{\vv}[1]{\vec{\mkern1mu#1}\!}

\newcommand{\U}{\mathcal U}
\newcommand{\V}{\mathcal V}

\newcommand\w{\omega}
\newcommand{\IN}{\mathbb N}

\newtheorem{theorem}{Theorem}
\newtheorem{problem}{Problem}
\newtheorem{lemma}{Lemma}

\newtheorem{claim}{Claim}

\theoremstyle{definition}
\newtheorem{definition}{Definition}
\newtheorem{remark}{Remark}

\title[The closedness of complete subsemilattices]{The closedness of complete subsemilattices in functionally Hausdorff semitopological semilattices}
\author{Taras Banakh, Serhi\u\i\ Bardyla, Alex Ravsky}
\address{T.Banakh: Ivan Franko National University of Lviv (Ukraine) and Jan Kochanowski University in Kielce (Poland)}
\email{t.o.banakh@gmail.com}
\address{S.~Bardyla: Institute of Mathematics, Kurt G\"{o}del Research Center, Vienna (Austria)}
\thanks{The second author was supported by the Austrian Science Fund FWF (Grant  I 3709-N35).}
\address{A.Ravsky: Ya.Pidstryhach Institute for Applied Problems of Mechanics and Mathematics, Lviv, Ukraine}
\email{alexander.ravsky@uni-wuerzburg.de}
\begin{document}

\begin{abstract} A topologized semilattice $X$ is {\em complete} if each non-empty chain $C\subset X$ has $\inf C\in\bar C$ and $\sup C\in\bar C$. It is proved that for any complete subsemilattice $X$ of a functionally Hausdorff semitopological semilattice $Y$ the partial order $P=\{(x,y)\in X\times X:xy=x\}$ of $X$ is closed in $Y\times Y$ and hence $X$ is closed in $Y$. This implies that for any continuous homomorphism $h:X\to Y$ from a compete topologized semilattice  $X$ to a functionally Hausdorff semitopological semilattice $Y$ the image $h(X)$ is closed in $Y$. The functional Hausdorffness of $Y$ in these two results can be replaced by the weaker separation axiom $\vv{T}_{2\delta}$, defined in this paper.
\end{abstract}
\maketitle

In this paper we continue to study the closedness properties of complete semitopological semilattices, which were introduced and studied by the authors in \cite{BBm}, \cite{BBc}, \cite{BBw}, \cite{BBseq}, \cite{BBo}. It turns out that complete semitopological semilattices share many common properties with compact topological semilattices, in particular their continuous homomorphic images in Hausdorff topological semilattices are closed.

A {\em semilattice} is any commutative semigroup of idempotents (an element $x$ of a semigroup is called an {\em idempotent} if $xx=x$). 

Each semilattice carries a natural partial order $\le$ defined by $x\le y$ iff $xy=x=yx$. Many properties of semilattices are defined in the language of this partial order. In particular, for a point $x\in X$ we can consider its upper and lower sets $${\uparrow}x:=\{y\in X:xy=x\}\mbox{ \ and \ }{\downarrow}x:=\{y\in X:xy=y\}$$ in the partially ordered set $(X,\le)$.


A subset $C$ of a semilattice $X$ is called a {\em chain} if $xy\in\{x,y\}$ for any $x,y\in C$. A semilattice $X$ is called {\em chain-finite} if each chain in $X$ is finite.

A semilattice endowed with a topology is called a {\em topologized semilattice}. 
A topologized semilattice $X$ is called a ({\em semi\,}){\em topological semilattice} if the semigroup operation $X\times  X\to X$, $(x,y)\mapsto xy$, is (separately) continuous. 

In \cite{Stepp1975} Stepp proved that for any homomorphism $h:X\to Y$ from a chain-finite semilattice to a Hausdorff topological semilattice $Y$, the image $h(X)$ is closed in $Y$. In \cite{BBm}, the first two authors improved this result of Stepp proving the following theorem.

\begin{theorem}[Banakh, Bardyla]\label{t:cf}  For any homomorphism $h:X\to Y$ from a chain-finite semilattice to a  Hausdorff semitopological semilattice $Y$, the image $h(X)$ is closed in $Y$.
\end{theorem} 

Topological generalizations of the notion of chain-finiteness are the notions of  chain-compactness and completeness, discussed in \cite{BBo}.

A topologized semilattice $X$ is called
\begin{itemize}
\item {\em chain-compact} if each closed chain in $X$ is compact;
\item {\em complete} if each non-empty  chain $C\subset X$ has $\inf C\in\bar C$ and $\sup C\in\bar C$.
\end{itemize}
Here $\bar C$ stands for the closure of $C$ in $X$.  Chain-compact and complete topologized semilattices appeared to be very helpful in studying the closedness properties of topologized semilattices, see \cite{BBm}, \cite{BBc}, \cite{BBw}, \cite{BBseq}, \cite{BBo}. By Theorem 3.1 \cite{BBm}, a Hausdorff semitopological semilattice is chain-compact if and only if complete (see also Theorem 4.3 \cite{BBo} for generalization of this characterization to topologized posets).
In \cite{BBm} the first two authors proved the following closedness property of complete topologized semilattices.

\begin{theorem}[Banakh, Bardyla]\label{t:kc} For any continuous homomorphism $h:X\to Y$ from a complete topologized semilattice $X$ to a Hausdorff topological semilattice $Y$, the image $h(X)$ is closed in $Y$.
\end{theorem}

Theorems~\ref{t:cf} and \ref{t:kc} motivate the following (still) open problem.

\begin{problem}\label{prob:main} Assume that $h:X\to Y$ is a continuous homomorphism from a complete topologized semilattice $X$ to a Hausdorff semitopological semilattice $Y$. Is $h(X)$ closed in $Y$?
\end{problem}

In \cite{BBseq} the first two authors authors gave the following partial answer to Problem~\ref{prob:main}.

\begin{theorem}[Banakh, Bardyla] For any continuous homomorphism $h:X\to Y$ from a complete topologized semilattice $X$ to a sequential Hausdorff semitopological semilattice $Y$, the image $h(X)$ is closed in $Y$.
\end{theorem}
  
In this paper we shall show that the answer to Problem~\ref{prob:main} is affirmative under the additional condition that the semitopological semilattice $Y$ satisfies the separation axiom $\vec{T}_{2\delta}$, which is stronger that the Hausdorff axiom $T_2$ and weaker than the axioms $T_3$ of regularity and $\vv{T}_{3\frac12}$ of functional Hausdorffness.

We recall that a topological space $X$ is defined to satisfy the separation axiom
\begin{itemize}
\item $T_1$ if for any distinct points $x,y\in X$ there exists an open set $U\subset X$ such that $x\in U$ and $y\notin U$;
\item $T_2$ if for any distinct points $x,y\in X$ there exists an open set $U\subset X$ such that $x\in U\subset\bar U\subset X\setminus\{y\}$;
\item $T_3$ if $X$ is a $T_1$-space and for any open set $V\subset X$ and point $x\in V$ there exists an open set $U\subset X$ such that $x\in U\subset\bar U\subset V$;
\item $T_{3\frac12}$ if $X$ is a $T_1$-space and for any open set $V\subset X$ and point $x\in V$ there exists a continuous function $f:X\to[0,1]$ such that $x\in f^{-1}([0,1))\subset V$.
\end{itemize}
In these definitions by $\bar A$ we denote the closure of a subset $A$ in a topological space.

Topological spaces satisfying a separation axiom $T_i$ are called {\em $T_i$-spaces}; $T_2$-spaces are called {\em Hausdorff}.
Now we define a new separation axiom.

\begin{definition} A topological space $X$ is defined to satisfy the separation axiom
\begin{itemize}
\item $T_{2\delta}$ if $X$ is a $T_1$-space and for any open set $U\subset X$ and point $x\in U$ there exists a countable family $\U$ of closed neighborhoods of $x$ in $X$ such that $\bigcap\U\subset U$.
\end{itemize}
\end{definition}
We shall say that a topological space $X$ satisfies the separation axiom 
\begin{itemize}
\item $\vv{T}_i$ for $i\in\{1,2,2\delta,3,3\frac12\}$ if $X$ admits a bijective continuous map $X\to Y$ to a $T_i$-space $Y$.
\end{itemize}

The following diagram describes the implications between the separation axioms $T_i$ and $\vv{T}_i$ for $i\in\{1,2,3,3\frac12\}$.
$$\xymatrix{
T_{3\frac12}\ar@{=>}[r]\ar@{=>}[d]&T_3\ar@{=>}[r]\ar@{=>}[d]&T_{2\delta}\ar@{=>}[r]\ar@{=>}[d]&T_2\ar@{=>}[r]\ar@{<=>}[d]&T_1\ar@{<=>}[d]\\
\vv{T}_{3\frac12}\ar@{=>}[r]&\vv{T}_3\ar@{=>}[r]&\vv{T}_{2\delta}\ar@{=>}[r]&\vv{T}_2\ar@{=>}[r]&\vv{T}_1
}
$$
Observe that a topological space $X$ satisfies the separation axiom $\vv{T}_{3\frac12}$ if and only if it is {\em functionally Hausdorff} in the sense that for any distinct points $x,y\in X$ there exists a continuous function $f:X\to\mathbb R$ with $f(x)\ne f(y)$. So, each functionally Hausdorff space is a $\vv{T}_{2\delta}$-space.

The following theorem gives a partial answer to Problem~\ref{prob:main} and is the main result of this paper.

\begin{theorem}\label{main} For any continuous homomorphism $h:X\to Y$ from a complete topologized semilattice $X$ to a semitopological semilattice $Y$ satisfying the separation axiom $\vv{T}_{2\delta}$, the image $h(X)$ is closed in $Y$.
\end{theorem} 

Theorem~\ref{main} will be proved in Section~\ref{s:m} after some preliminary work made in Sections~\ref{s:t}. As a by-product of the proof we obtain the  
following useful result.

\begin{theorem}\label{t:order} For any complete semitopological semilattice $X$ satisfying the separation axiom $\vv{T}_{2\delta}$, the partial order $\le:=\{(x,y)\in X\times X:xy=x\}$ of $X$ is a closed subset of $X\times X$.
\end{theorem}

\begin{remark} The completeness of $X$ is essential in Theorem~\ref{t:order}: by \cite{BBR2}, there exists a metrizable semitopological semilattice $X$ whose partial order is not closed in $X\times X$, and for every $x\in X$ the upper set ${\uparrow}x$ is finite.
\end{remark}

\section{The closedness of the partial order of a semitopological semilattice}\label{s:t}

In this section we shall prove Theorem~\ref{t2} implying Theorem~\ref{t:order}. In the proof we shall use the following lemma. In this lemma we identify cardinals with the smallest ordinals of a given cardinality.

\begin{lemma}\label{l:smallest} Each non-empty subsemilattice $S$ of a complete topologized semilattice $X$ has $\inf S\in\bar S$.
\end{lemma}

\begin{proof} By transfinite induction, we shall prove that for any non-zero cardinal $\kappa$ the following condition holds:
\begin{itemize}
\item[$(*_\kappa)$] each non-empty semilattice $S\subset X$ of cardinality $|S|\le\kappa$ has $\inf S\in \bar S$.
\end{itemize}

For any finite cardinal $\kappa$ the statement $(*_\kappa)$ is true (as each finite semilattice contains the smallest element). Assume that for some infinite cardinal $\kappa$ and any non-zero cardinal $\lambda<\kappa$ the statement $(*_\lambda)$ has been proved. To prove the statement $(*_\kappa)$, choose any subsemilattice $S\subset X$ of cardinality $|S|=\kappa$. Write $S$  as $S=\{x_\alpha\}_{\alpha\in\kappa}$. For every  ordinal $\beta\in\kappa$ consider the subsemilattice $S_\beta$ generated by the set $\{x_\alpha:\alpha\le\beta\}$. Since $\lambda:=|S_\beta|\le|\w+\beta|<\kappa$, by the condition $(*_\lambda)$, the semilattice $S_\beta$ has $\inf S_\beta\in\bar S_\beta\subset \bar S$. For any ordinals $\alpha<\beta$ in $\kappa$ the inclusion $S_\alpha\subset S_\beta$ implies $\inf S_\beta\le\inf S_\alpha$. By the $k$-completeness of $X$, the chain $C=\{\inf S_\alpha:\alpha\in\kappa\}\subset \bar S$ has $\inf C\in\bar C\subset  \bar S$. We claim that $\inf C=\inf S$. Indeed, for any $\alpha\in\kappa$ we get $\inf C\le\inf S_\alpha\le x_\alpha$, so $\inf C$ is a lower bound for the set $S$ and hence $\inf C\le\inf S$. On the other hand, for any lower bound $b$ of $S$ and every $\alpha\in\kappa$, we get $S_\alpha\subset S\subset{\uparrow}b$ and hence $b\le \inf S_\alpha$. Then $\inf C=\inf\{\inf S_\alpha:\alpha\in\kappa\}\ge b$ and finally $\inf C=\inf S$.
\end{proof}

The following theorem is technically the most difficult result of this paper.

\begin{theorem}\label{t:Tdelta} Let $X$ be a complete subsemilattice of a semitopological semilattice $Y$ satisfying the separation axiom ${T}_{2\delta}$. Then the partial order $P:=\{(x,y)\in X\times X:xy=x\}$ of $X$ is closed in $Y\times Y$.
\end{theorem}

\begin{proof} Given two points $x,y\in Y$ we should prove that $(x,y)\in P$ if for any neighborhoods $O_x$ and $O_y$ of $x,y$ in $Y$ there are points $x'\in O_x\cap X$ and $y'\in O_y\cap X$ with $x'\le y'$. By transfinite induction, for every infinite cardinal $\kappa$ we shall prove the following statement:

\begin{itemize}
\item[$(*_\kappa)$] for any families $\U_x$, $\U_y$ of closed neighborhoods of $x$ and $y$ with $\max\{|\U_x|,|\U_y|\}\le\kappa$ there are points $x'\in X\cap\bigcap\U_x$ and $y'\in X\cap\bigcap\U_y$ such that $x'\le y'$.
\end{itemize}

\begin{claim}\label{cl1} The statement $(*_\w)$ holds.
\end{claim}

\begin{proof} Fix any countable families $(U_n)_{n\in\w}$ and $(V_n)_{n\in\w}$ of closed neighborhoods of the points $x,y$, respectively. Replacing each set $U_n$ by $\bigcap_{i\le n}U_i$, we can assume that $U_{n+1}\subset U_n$ for all $n\in\w$. By the same reason, we can assume that the sequence $(V_n)_{n\in\w}$ is decreasing. For every $n\in\w$ denote by $U_n^\circ$ and $V_n^\circ$ the interiors of the sets $U_n$ and $V_n$ in $Y$.

By induction we shall construct sequences $(x_n)_{n\in\w}$ and $(y_n)_{n\in\w}$ of points of $X$ such that for every $n\in\w$ the following conditions are satisfied:
\begin{enumerate}
\item[$(1_n)$] $x_n\le y_n$;
\item[$(2_n)$] $\{x_i\cdots x_n,x_i\cdots x_nx\}\subset U_i^\circ$ for all $i\le n$;
\item[$(3_n)$] $\{y_i\cdots y_n,y_i\cdots y_ny\}\subset V_i^\circ$ for all $i\le n$.
\end{enumerate}

To choose the initial points $x_0,y_0$, find neighborhoods $U'_0\subset U_0^\circ$ and $V'_0\subset V_0^\circ$ of $x$ and $y$ in $Y$ such that $U'_0x\subset U_0^\circ$ and $V'_0y\subset V_0^\circ$. By our assumption, there are points $x_0\in X\cap U'_0$ and $y_0\in X\cap V'_0$ such that  $x_0\le y_0$. The choice of the neighborhoods $U'_0$ and $V'_0$ ensures that the conditions $(2_0)$ and $(3_0)$ are satisfied.

Now assume that for some $n\in\IN$ points $x_0,\dots,x_{n-1}$ and $y_0,\dots,y_{n-1}$ of $X$ are chosen so that the conditions $(1_{n-1})$--$(3_{n-1})$ are satisfied. The condition $(2_{n-1})$ implies that for every $i\le n$ we have the inclusion $x_i\cdots x_{n-1}xx=x_i\cdots x_{n-1}x\in U_i^\circ$ (if $i=n$, then we understand that $x_i\cdots x_{n-1}x=x$).
Using the continuity of the shift $s_x:Y\to Y$, $s_x:z\mapsto xz$, we can  find a neighborhood $U_{n}'\subset Y$ of $x$ such that $x_i\cdots x_{n-1}\cdot(U_n'\cup U_n'x)\subset U_i^\circ$ for every $i\le n$. By analogy, we can find a neighborhood $V_n'\subset Y$ of $y$ such that   $y_i\cdots y_{n-1}\cdot(V_n'\cup V_n'y)\subset V_i^\circ$ for every $i\le n$.
By our assumption, there are points $x_n\in X\cap U_n'$ and $y_n\in X\cap V_n'$ such that $x_n\le y_n$. The choice of the neighborhoods $U'_0$ and $V'_0$ ensures that the conditions $(2_n)$ and $(3_n)$ are satisfied.
This completes the inductive step.
\smallskip

Now for every $i\in\w$ consider the chain $C_i=\{x_i\cdots x_n:n\ge i\}\subset U_i^\circ$ in $X$.  By the completeness of $X$, this chain has $\inf C_i\in X\cap\bar C_i\subset X\cap\overline{U_i^\circ}\subset X\cap U_i$. Observing that $\inf C_i\le x_ix_{i+1}\cdots x_n\le x_{i+1}\cdots x_n$ for all $i>n$, we see that $\inf C_i$ is a lower bound of the chain $C_{i+1}$ and hence $\inf C_i\le\inf C_{i+1}$. By the completeness of $X$, for every $i\in\w$ the chain $D_i:=\{\inf C_j:j\ge i\}\subset U_i$ has $\sup D_i\in X\cap\bar D_i\subset X\cap  U_i$. Since the sequence $(\inf C_i)_{i\in\w}$ is increasing, we get $\sup D_0=\sup D_i\in X\cap U_i$ for all $i\in\w$. Consequently, $\sup D_0\in X\cap\bigcap_{i\in\w}U_i$.

By analogy, for every $k\in\w$ consider the chain $E_i=\{y_i\cdots y_n:n\ge i\}\subset V_i^\circ$ in $X$.  By the completeness of $X$, this chain has $\inf E_i\in X\cap\bar E_i\subset X\cap\overline{V_i^\circ}\subset X\cap V_i$. By the completeness of $X$, for every $i\in\w$ the chain $F_i:=\{\inf E_j:j\ge i\}\subset V_i$ has $\sup F_i\in X\cap\bar F_i\subset X\cap \overline{V_i}=X\cap V_i$. Since the sequence $(\inf E_i)_{i\in\w}$ is increasing, we get $\sup F_0=\sup F_i\in X\cap V_i$ for all $i\in\w$. Consequently, $\sup F_0\in X\cap\bigcap_{i\in\w}V_i$.

To finish the proof of Claim~\ref{cl1}, it suffices to show that $\sup D_0\le \sup F_0$. The inductive conditions $(1_n)$, $n\in\w$, imply that $\inf C_i\le \inf E_i$ for all $i\in\w$ and $\sup D_0=\sup\{\inf C_i:i\in\w\}\le \sup\{\inf E_i:i\in\w\}=\sup F_0$.
\end{proof}

Now assume that some uncountable cardinal $\kappa$ we have proved that the property $(*_\lambda)$ holds for any infinite cardinal $\lambda<\kappa$. To prove the proporty $(*_\kappa)$, fix families $(U_\alpha)_{\alpha\in\kappa}$ and $(V_{\alpha})_{\alpha\in\kappa}$ of closed neighborhoods of $x$ and $y$ in $Y$, respectively.

By transfinite induction, we shall construct sequences $(x_\alpha)_{\alpha\in\kappa}$ and $(y_\alpha)_{\alpha\in\kappa}$ of points of $X$ such that for every $\alpha\in\kappa$ the following conditions are satisfied:
\begin{enumerate}
\item[$(1_\alpha)$] $x_\alpha\le y_\alpha$;
\item[$(2_\alpha)$] $\{x_{\alpha_1}\cdots x_{\alpha_n},x_{\alpha_1}\cdots x_{\alpha_n}x\}\subset U_{\alpha_0}^\circ$ for any sequence of ordinals $\alpha_0<\alpha_1<\dots<\alpha_n=\alpha$ with $n\in\IN$;
\item[$(3_\alpha)$] $\{y_{\alpha_1}\cdots y_{\alpha_n},y_{\alpha_1}\cdots y_{\alpha_n}y\}\subset V_{\alpha_0}^\circ$ for any sequence of ordinals $\alpha_0<\alpha_1<\dots<\alpha_n=\alpha$ with $n\in\IN$.
\end{enumerate}
To start the inductive construction, choose neighborhoods $U_0',V_0'\subset Y$ of the points $x,y$ such that $U'_0\cup U_0'x\subset U_0^\circ$ and $V_0'\cup V_0'y\subset V_0$. By our assumption there are points $x_0\in X\cap U_0'$ and $y_0\in X\cap V_0'$ such that $x_0\le y_0$. 

Now assume that for some ordinal $\beta\in\kappa$ we have constructed points $x_\alpha,y_\alpha$ for all ordinals $\alpha<\beta$ so that the conditions $(1_\alpha)$--$(3_\alpha)$ are satisfied. For any $n\in\IN$ and any sequence of ordinals $\alpha_0<\dots<\alpha_{n}=\beta$, use the inductive condition $(2_{\alpha_{n-1}})$ and the continuity of the shift $s_x:Y\to Y$ for finding a neighborhood $U_{\alpha_0,\dots,\alpha_n}\subset Y$ of $x$ such that 
$$x_{\alpha_1}\cdots x_{\alpha_{n-1}}U_{\alpha_0,\dots,\alpha_n}\cup   x_{\alpha_1}\cdots x_{\alpha_{n-1}}U_{\alpha_0,\dots,\alpha_n}x\subset U_{\alpha_0}^\circ.$$If $n=1$, then we assume that $x_{\alpha_1}\cdots x_{\alpha_{n-1}}U_{\alpha_0,\dots,\alpha_n}=U_{\alpha_0,\dots,\alpha_{n}}$.

In similar way choose a neighborhood $V_{\alpha_0,\dots,\alpha_n}\subset Y$ of $y$ such that 
$$y_{\alpha_1}\cdots y_{\alpha_{n-1}}V_{\alpha_0,\dots,\alpha_n}\cup   y_{\alpha_1}\cdots y_{\alpha_{n-1}}V_{\alpha_0,\dots,\alpha_n}y\subset V_{\alpha_0}^\circ.$$

Since $Y$ is a $T_{2\delta}$-space, there are countable families $\U'_{\alpha_0,\dots,\alpha_n}$ and  $\V'_{\alpha_0,\dots,\alpha_n}$ of closed neighborhoods of $x$ and $y$ respectively such that $\bigcap\U'_{\alpha_0,\dots,\alpha_n}\subset U_{\alpha_0,\dots,\alpha_n}$ and $\bigcap\V'_{\alpha_0,\dots,\alpha_n}\subset V_{\alpha_0,\dots,\alpha_n}$. Also choose countable families $\U_\beta'$ and $\V_\beta'$ of closed neighborhoods of $x$ and $y$ in $Y$ such that $\bigcap\U_\beta'\subset U_\beta^\circ$ and $\bigcap\V_\beta'\subset V_\beta^\circ$.

Now consider the families $$\U_\beta=\bigcup_{n=0}^\infty\{\U'_{\alpha_0,\dots,\alpha_n}:\alpha_0<\dots<\alpha_n=\beta\}\mbox{ and }
\V_\beta=\bigcup_{n=0}^\infty\{\V'_{\alpha_0,\dots,\alpha_n}:\alpha_0<\dots<\alpha_n=\beta\}$$ and observe that $\lambda:=\max\{|\U_\beta|,|\V_\beta|\}\le|\w+\beta|<\kappa$. By the inductive assumption, the condition $(*_\lambda)$ holds. So we can find points $x_\alpha\in X\cap\bigcap\U_\beta$ and $y_\beta\in X\cap \bigcap\V_\beta$ such that $x_\alpha\le y_\beta$. Observe that for any $n\in \IN$ and any sequence of ordinals $\alpha_0<\dots<\alpha_n=\beta$ we have $$x_{\alpha_1}\cdots x_{\alpha_n}\in x_{\alpha_1}\cdots x_{\alpha_{n-1}}\cdot \textstyle{\bigcap}\U_\beta\subset  x_{\alpha_1}\cdots x_{\alpha_{n-1}}\cdot \bigcap\U'_{\alpha_0,\dots,\alpha_n}\subset x_{\alpha_1}\cdots x_{\alpha_{n-1}}\cdot U_{\alpha_0,\dots,\alpha_n}\subset U_{\alpha_0}^\circ.$$ By analogy we can prove that $x_{\alpha_1}\cdots x_{\alpha_n}x\in U_{\alpha_0}^\circ$, which means that the inductive condition $(2_\beta)$ is satisfied. By analogy we can check that the inductive condition $(3_\beta)$ is satisfied. This completes the inductive step.
\smallskip

After completing the inductive construction, for every $\alpha\in\kappa$ consider the subsemilattices $$C_\alpha:=\{x_{\alpha_0}\cdots x_{\alpha_n}:\alpha<\alpha_0<\dots<\alpha_n<\kappa\}\mbox{ \ and \ }E_\alpha=\{y_{\alpha_0}\cdots y_{\alpha_n}:\alpha<\alpha_0<\dots<\alpha_n<\kappa\}$$ of $X$. The inductive conditions $(2_\beta)$ and $(3_\beta)$ for $\beta\ge\alpha$ imply that $C_\alpha\subset U_\alpha^\circ\subset U_\alpha$ and $E_\alpha\subset V_\alpha^\circ\subset V_\alpha$. By Lemma~\ref{l:smallest}, the semilattices $C_\alpha$ and $E_\alpha$ have $\inf C_\alpha\in X\cap \bar C_\alpha\subset X\cap U_\alpha$ and $\inf E_\alpha\in X\cap \bar E_\alpha\subset X\cap V_\alpha$. For any ordinals $\alpha<\beta$ in $\kappa$ the inclusions $C_\beta\subset C_\alpha$ and $E_\beta\subset E_\alpha$ imply that the transfinite sequences $(\inf C_\alpha)_{\alpha\in\kappa}$ and $(\inf E_\alpha)_{\alpha\in\kappa}$ are non-decreasing. By the completeness of $X$ for every $\alpha\in\kappa$ the chains $D_\alpha=\{\inf C_\beta:\alpha\le\beta<\kappa\}\subset X\cap U_\alpha$ and $F_\alpha:=\{\inf E_\beta:\alpha\le\beta<\kappa\}\subset X\cap V_\alpha$ have $\sup D_\alpha\in X\cap\bar D_\alpha\subset X\cap U_\alpha$ and  
 $\sup F_\alpha\in X\cap\bar F_\alpha\subset X\cap V_\alpha$. Taking into account that the transfinite sequences $(\inf C_\alpha)_{\alpha\in\kappa}$ and $(\inf E_\alpha)_{\alpha\in\kappa}$ are non-decreasing, we conclude that $\sup D_0=\sup D_\alpha\in X\cap U_\alpha$ and $\sup F_0=\sup E_\alpha\in V_\alpha$ for all $\alpha\in\kappa$. Consequently $\sup D_0\in X\cap\bigcap_{\alpha\in\kappa}U_\alpha$ and $\sup F_0\in X\cap\bigcap_{\alpha\in\kappa}V_\alpha$.

To finish the proof of the property $(*_\kappa)$, it suffices to show that $\sup D_0\le \sup F_0$. The inductive conditions $(1_\alpha)$, $\alpha\in\kappa$, imply that $\inf C_\alpha\le \inf E_\alpha$ for all $\alpha\in\w$ and then $$\sup D_0=\sup\{\inf C_\alpha:\alpha\in\kappa\}\le \sup\{\inf E_\alpha:\alpha\in\kappa\}=\sup F_0.$$
By the Principle of Transfinite Induction, the property $(*_\kappa)$ holds for all infinite cardinals $\kappa$. 
\smallskip

Denote by $\U_x$ and $\U_y$ the families of closed neighborhoods of the points $x$ and $y$, respectively. Taking into account that the space $Y$ is Hausdorff, we conclude that $\{x\}=\bigcap\U_x$ and $\{y\}=\bigcap\U_y$. The property $(*_\kappa)$ for $\kappa=\max\{|\U_x|,|\U_y|\}$ implies that there are points $x'\in X\cap\bigcap\U_x=X\cap\{x\}$ and $y'\in X\cap\bigcap\U_y=X\cap\{y\}$ such that $x'\le y'$. It follows that $x=x'\le y'=y$ and hence $(x,y)\in P$.
\end{proof}

Theorem~\ref{t:Tdelta} implies its own self-improvement.

\begin{theorem}\label{t2} Let $X$ be a complete subsemilattice of a semitopological semilattice $Y$ satisfying the separation axiom $\vv{T}_{2\delta}$. Then the partial order $P:=\{(x,y)\in X\times X:xy=x\}$ of $X$ is closed in $Y\times Y$ and the semilattice $X$ is closed in $Y$.
 \end{theorem}
 
 \begin{proof} Since $Y$ satisfies the separation axiom $\vv{T}_{2\delta}$ and the class of $T_{2\delta}$-spaces is closed under taking subspaces and Tychonoff products, the space $Y$ admits a weaker topology $\tau$ such that  $(Y,\tau)$ is a $T_{2\delta}$-space and any continuous map $Y\to Z$ to a $T_{2\delta}$-space $Z$ remains continuous in the topology $\tau$. This implies that for every $a\in Y$ the shift $s_a:Y\to Y$, $s_a:y\mapsto ay$, is continuous with respect to the topology $\tau$ and hence $Y$ endowed with the topology $\tau$ is a semitopological semilattice.

Since the identity map $Y\to(Y,\tau)$ is continuous, the completeness of the subsemilattice $X\subset Y$ implies the completeness of $X$ with respect to the subspace topology inherited from $\tau$. Applying Theorem~\ref{t:Tdelta}, we conclude that the partial order $P=\{(x,y)\in X\times X:xy=x\}$ is closed in the square of the space $(Y,\tau)$ and hence is closed in the original topology of $Y$ (which is stronger than $\tau$).

By Lemma~\ref{l:smallest}, the semilattice $X$ has the smallest element $s\in X$.
Observing that $X=\{x\in Y:(s,x)\in P\}$, we see that the set $X$ is closed in $Y$.
\end{proof}
 
\section{Proof of Theorem~\ref{main}}\label{s:m}

Let $h:X\to Y$ be a continuous homomorphism from a complete topologized semilattice $X$ to a semitopological semilattice $Y$ satisfying the separation axiom $\vv{T}_{2\delta}$. We need to show that the subsemilattice $h(X)$ is closed in $Y$. This will follow from Theorem~\ref{t2} as soon as we show that the semitopologhical semilattice $h(X)$ is complete.

Given any non-empty chain $C\subset h(X)$, we should show that $C$ has $\inf C\in\bar C$ and $\sup C\in\bar C$.

For every $c\in C$ consider the closed subsemilattice $S_c:=h^{-1}({\uparrow}c)$ in $X$. By Lemma~\ref{l:smallest} the semilattice $S_c$ has $\inf S_c\in\bar S_c=S_c$, which means that $\inf S_c$ is the smallest element of $S_c$.  It follows from $h(S_c)={\uparrow}c$ that $h(\inf S_c)\in h(S_c)={\uparrow}c$ and hence $c\le h(\inf S_c)$. On the other hand, for any $x\in h^{-1}(c)$ we get $\inf S_c\le x$ and hence $h(\inf S_x)\le h(x)=c$ and finally $h(\inf S_c)=c$. 

It is clear that for any $x\le y$ in $C$, we get $S_x\subset S_y$ and hence $\inf S_x\le\inf S_y$. This means that $D=\{\inf S_x:x\in C\}$ is a chain in $X$. By the completeness of $X$, the chain $D$ has $\inf D\in\bar D$ and $\sup D\in\bar D$. The continuity of the homomorphism $h$ implies that $h(\inf D)\in h(\bar D)\subset \overline{h(D)}=\bar C$ and $h(\sup D)\in h(\bar D)\subset \overline{h(D)}=\bar C$.
It remains to show that $h(\inf D)=\inf C$ and $h(\sup D)=\sup C$.

Taking into account that $h$ is a semilattice  homomorphism, we can show that $h(\inf D)$ is a lower bound of the set $C=h(D)$ in $Y=h(X)$. For any other lower bound $b\in h(X)$ of $C$, we see that $C\subset{\uparrow} b$, $D\subset h^{-1}({\uparrow}b)=S_b$ and hence $\inf S_b\le \inf D$, which implies $b=h(\inf S_b)\le h(\inf D)$. So, $\inf C=h(\inf D)$. By analogy we can prove that $\sup C=h(\sup D)$.

\end{document}